\documentclass[a4paper, 12pt]{article}
\usepackage{}

\usepackage{mathrsfs}
\usepackage{epsfig}
\usepackage{amsmath}
\usepackage{amssymb}
\usepackage{latexsym}
\usepackage{amsfonts}
\usepackage{amsthm}

\usepackage[numbers,sort&compress]{natbib}
\usepackage{graphicx}
\ifx\pdfoutput\undefined  \else
 \fi

\usepackage[centerlast]{caption2}

\usepackage{color}

\usepackage{txfonts}
\usepackage{amssymb}
\usepackage{mathrsfs}
\usepackage{amsfonts}

\newtheorem{theorem}{Theorem}[section]
\newtheorem{lemma}[theorem]{Lemma}

\newtheorem{conj}[theorem]{Conjecture}

\usepackage{latexsym,amssymb}
\usepackage{amsmath}

\newcommand{\F}{{\mathbb F}}

\begin{document}

\title{On the modular Erd\H{o}s-Burgess constant}

\author{
Jun Hao$^{a}$  \ \ \ \ \ \
Haoli Wang$^{b}$\thanks{Corresponding author's Email:
bjpeuwanghaoli@163.com}  \ \ \ \ \ \
Lizhen Zhang$^{a}$
\\
\\
$^{a}${\small Department of Mathematics, Tianjin Polytechnic University, Tianjin, 300387, P. R. China}\\
\\
$^{b}$ {\small College of Computer and Information Engineering,}\\
{\small Tianjin Normal University, Tianjin, 300387, P. R. China}\\
}

\date{}
\maketitle

\begin{abstract} Let $n$ be a positive integer. For any integer $a$, we say that $a$ is idempotent modulo $n$ if $a^2\equiv a\pmod n$. The $n$-modular Erd\H{o}s-Burgess constant is the
smallest positive integer $\ell$  such that any $\ell$ integers contain one or more integers whose product
is idempotent modulo $n$.
We gave a sharp lower bound of the $n$-modular Erd\H{o}s-Burgess constant,
in particular, we
determined the $n$-modular Erd\H{o}s-Burgess constant in the case when $n$ is a prime power or a product of pairwise distinct primes.
\end{abstract}


\section {Introduction}

Let $\mathcal{S}$ be a finite multiplicatively written  commutative semigroup with identity $1_{\mathcal{S}}$. By a sequence over
$\mathcal{S}$, we mean a finite unordered sequence of terms from $\mathcal{S}$ where repetition is allowed. For a sequence $T$
over $\mathcal{S}$ we denote by $\pi(T)\in \mathcal{S}$ the product of its terms and we say that $T$ is a product-one sequence
if $\pi(T) = 1_{\mathcal{S}}$. If $\mathcal{S}$ is a finite abelian group, the Davenport constant ${\rm D}(\mathcal{S})$ of $\mathcal{S}$ is the smallest positive integer $\ell$  such that every sequence $T$
over $\mathcal{S}$ of length $|T|\geq \ell$ has a nonempty product-one subsequence. The Davenport constant has mainly
been studied for finite abelian groups but also in more general settings (we refer to \cite{GH, GeroldingerRuzsa, Grynkiewiczmono, taovu, GaoGero}  for work in the setting of abelian groups, to \cite{CziszterDomoGero,GaoLiPeng} for work in case of
non-abelian groups, and to \cite{wangDavenportII,wangAddtiveirreducible, wanggao,gaowangII,wang-zhang-qu} for work in commutative semigroups).

In the present paper we study the Erd\H{o}s-Burgess constant $\textsc{I}(\mathcal{S})$ of $\mathcal{S}$ which is defined as the smallest positive integer
$\ell$ such that every sequence $T$ over $\mathcal{S}$ of length $|T|\geq \ell$ has a non-empty subsequence $T'$
whose product $\pi(T')$ is an idempotent of $\mathcal{S}$. Clearly, if $\mathcal{S}$ happens to be a finite abelian group, then the unique idempotent of $\mathcal{S}$ is the identity  $1_{\mathcal{S}}$, whence $\textsc{I}(\mathcal{S}) = {\rm D}(\mathcal{S})$. The study of $\textsc{I}(\mathcal{S})$ for general semigroups was initiated by a question
of Erd\H{o}s and has found renewed attention in recent
years (e.g., \cite{Burgess69,Gillam72,wangStruucture,wangErdos-burgess,wang-Hao-zhang}).
For a commutative unitary ring $R$, let $\mathcal{S}_R$ be the multiplicative semigroup of the ring $R$, and $R^{\times}$ the group of units of $R$. Notice that the group $R^{\times}$ is a subsemigroup of the semigoup $\mathcal{S}_R$. We state our main result.

\begin{theorem}\label{Theorem main}
\  Let $n>1$ be an integer, and let $R=\mathbb{Z}_n$ be the ring of integers modulo $n$.
Then
$${\rm I}(\mathcal{S}_R)\geq {\rm D}(R^{\times})+\Omega(n)-\omega(n),$$ where $\Omega(n)$ is the number of primes occurring in the prime-power decomposition of $n$ counted with
multiplicity, and $\omega(n)$ is the number of distinct primes.  Moreover, if $n$ is a prime power
or a product of pairwise distinct primes, then equality holds.
\end{theorem}

\section{Notation}

Let $\mathcal{S}$ be a finite  multiplicatively written
commutative semigroup with the binary operation *. An element $a$ of $\mathcal{S}$ is said to be idempotent if $a*a=a$. Let
${\rm E}(\mathcal{S})$ be the set of idempotents of $\mathcal{S}$.
We introduce sequences over semigroups and follow the notation and terminology of Grynkiewicz and
others (cf. [\cite{Grynkiewiczmono}, Chapter 10] or \cite{CziszterDomoGero, Grynkiewicz13}).
Sequences over $\mathcal{S}$ are considered as elements
in the free abelian monoid $\mathcal{F}(\mathcal{S})$ with basis $\mathcal{S}$.
In order to avoid confusion between the multiplication in $\mathcal{S}$ and multiplication in $\mathcal{F}(\mathcal{S})$,
we denote multiplication in $\mathcal{F}(\mathcal{S})$ by the boldsymbol $\cdot$ and we use brackets for all exponentiation in $\mathcal{F}(\mathcal{S})$.
In particular, a sequence $\mathcal{S}\in \mathcal{F}(\mathcal{S})$ has the form
\begin{equation}\label{equation form of sequences}
T=a_1a_2\cdot\ldots\cdot a_{\ell}=\mathop{\bullet}\limits_{i\in [1,\ell]} a_i=\mathop{\bullet}\limits_{a\in \mathcal{S}} a^{[{\rm v}_a(T)]}\in \mathcal{F}(\mathcal{S})
\end{equation}
where $a_1,\ldots,a_{\ell}\in \mathcal{S}$ are the terms of $T$, and ${\rm v}_a(T)$ is the multiplicity of the term $a$ in $T$. We call $|T|=\ell=\sum\limits_{a\in \mathcal{S}} {\rm v}_a(T)$ the length of $T$.
Moreover, if $T_1, T_2\in \mathcal{F}(\mathcal{S})$ and $a_1,a_2 \in \mathcal{S}$, then $T_1\cdot T_2\in \mathcal{F}(\mathcal{S})$
has length $|T_1|+|T_2|, T_1\cdot a_1 \in \mathcal{F}(\mathcal{S})$ has length $|T_1|+1$,  $a_1\cdot a_2\in \mathcal{F}(\mathcal{S})$
is a sequence of length $2$. If $a \in \mathcal{S}$ and $k\in \mathbb{N}_0$, then
$a^{[k]}=\underbrace{a\cdot\ldots\cdot a}_{k}\in \mathcal{F}(\mathcal{S}).$ Any sequence $T_1\in \mathcal{F}(\mathcal{S})$ is called a subsequence of $T$ if ${\rm v}_a(T_1)\leq {\rm v}_a(T)$ for every element $a\in \mathcal{S}$, denoted $T_1\mid T$. In particular,
if $T_1\neq T$, we call $T_1$ a {\sl proper} subsequence of $T$, and let $T\cdot   T_1^{[-1]}$ denote the sequence resulting by removing the terms of $T_1$ from $T$.

Let $T$ be a sequence as in \eqref{equation form of sequences}. Then
\begin{itemize}
     \item $\pi(T)=a_1*\cdots *a_{\ell}$ is the product of all terms of $T$, and
     \item $\prod(T)=\{\prod_{j\in J} a_j: \emptyset \neq J\subset [1,\ell]\}\subset\mathcal{S}$ is the set of subsequence products of $T$.
         \end{itemize}
         We say that $T$ is
\begin{itemize}
     \item a {\sl product-one sequence} if $\pi(T)=1_{\mathcal{S}},$
     \item an {\sl idempotent-product sequence} if $\pi(T)\in {\rm E}(\mathcal{S}),$
     \item {\sl product-one free} if $1_{\mathcal{S}}\notin \prod(T),$
     \item {\sl idempotent-product free} if ${\rm E}(\mathcal{S})\cap \prod(T)=\emptyset.$
\end{itemize}

Let $n>1$ be an integer. For any integer $a$, we denote $\bar{a}$ the congruence class of $a$ modulo $n$. Any integer $a$ is said to be {\sl idempotent modulo $n$} if $aa\equiv a\pmod n$, i.e., $\bar{a}\bar{a}=\bar{a}$ in $\mathbb{Z}_n$.
A sequence $T$ of integers is said to be {\sl idempotent-product free modulo $n$} provided that $T$ contains no nonempty subsequence $T'$ with $\pi(T')$ being idempotent modulo $n$. We remark that saying a sequence $T$ of integers is idempotent-product free modulo $n$ is equivalent to saying the sequence $\mathop{\bullet}\limits_{a\mid T} \bar{a}$ is idempotent-product free in the multiplicative semigroup of the ring $\mathbb{Z}_n$.

\section{Proof of Theorem \ref{Theorem main}}

\begin{lemma} \label{Lemma idempotent form} Let $n=p_1^{k_1}p_2^{k_2}\cdots p_r^{k_r}$ be a positive integer where $r\geq 1,$ $k_1,k_2,\ldots,k_r\geq 1$, and
$p_1, p_2, \ldots,p_r$ are distinct primes.
For any integer $a$, the congruence $a^2\equiv a\pmod n$ holds if and only if $a\equiv 0\pmod {p_i^{k_i}}$ or $a\equiv 1\pmod {p_i^{k_i}}$ for every $i\in [1,r]$.
\end{lemma}

\begin{proof} \ Note that $a^2\equiv a\pmod n$ if and only if ${p_i^{k_i}}$ divides $a(a-1)$ for all $i\in [1,r]$. Since $\gcd(a, a-1)=1$,  it follows that $a^2\equiv a\pmod n$ holds if and only if ${p_i^{k_i}}$ divides $a$ or $a-1$, i.e., $a\equiv 0\pmod {p_i^{k_i}}$ or $a\equiv 1\pmod {p_i^{k_i}}$ for every $i\in [1,r]$, completing the proof. \end{proof}

\noindent {\sl Proof of Theorem  \ref{Theorem main}.} \ Say
\begin{equation}\label{equation factorization of n}
n=p_1^{k_{1}}p_2^{k_{2}}\cdots p_r^{k_r},
\end{equation}
 where $p_1,p_2, \ldots,p_r$ are distinct primes and $k_i\geq 1$ for all $i\in [1,r]$.
 Observe that
 \begin{equation}\label{equation bigomega(K)}
 \Omega(n)=\sum\limits_{i=1}^r k_i
 \end{equation}
 and \begin{equation}\label{equation smallomega(K)}
 \omega(n)=r.
 \end{equation}

 Take a sequence $V$ of integers of of length ${\rm D}(R^{\times})-1$ such that \begin{equation}\label{equation bar V in Rtimes}
 \mathop{\bullet}\limits_{a\mid V} \bar{a}\in \mathcal{F}(R^{\times})
 \end{equation}
  and
 \begin{equation}\label{equation bar 1 not in prod}
 \bar{1}\notin\prod(\mathop{\bullet}\limits_{a\mid V} \bar{a}).
 \end{equation}
  Now we show that the sequence $V\cdot (\mathop{\bullet}\limits_{i\in [1,r]} p_i^{[k_i-1]})$ is idempotent-product free modulo $n$. Suppose to the contrary that $V\cdot (\mathop{\bullet}\limits_{i\in [1,r]} p_i^{[k_i-1]})$ contains a {\bf nonempty} subsequence $W$, say $W=V'\cdot (\mathop{\bullet}\limits_{i\in [1,r]} p_i^{[\beta_i]})$,  such that $\pi(W)$ is idempotent modulo $n$, where $V'$ is a subsequence of $V$ and $$\beta_i\in [0,k_i-1]\mbox{ for all } i\in [1,r].$$ It follows that
\begin{equation}\label{equation theta sigma(W)}
\pi(W)=\pi(V')p_1^{\beta_1}\cdots p_r^{\beta_r}.
\end{equation}
If $\sum\limits_{i=1}^r\beta_i=0$, then $W=V'$ is a {\sl nonempty} subsequence of $V$. By \eqref{equation bar V in Rtimes} and \eqref{equation bar 1 not in prod}, there exists some $t\in [1,r]$ such that $\pi(W)\not \equiv 0 \pmod {p_t^{k_t}}$ and $\pi(W)\not \equiv 1 \pmod {p_t^{k_t}}$. By Lemma \ref{Lemma idempotent form}, $\pi(W)$ is not idempotent modulo $n$,  a contradiction.
Otherwise, $\beta_j>0$ for some $j\in [1,r]$, say
\begin{equation}\label{equation beta1 in [1,n1-1]}
\beta_1\in [1,k_1-1].
\end{equation}
 Since $\gcd(\pi(V'),p_1)=1$, it follows from \eqref{equation theta sigma(W)} that $\gcd(\pi(W),p_1^{k_1})=p_1^{\beta_1}$. Combined with \eqref{equation beta1 in [1,n1-1]}, we have that $\pi(W)\not\equiv 0\pmod {p_1^{k_1}}$ and
$\pi(W)\not\equiv 1 \pmod {p_1^{k_1}}$. By
Lemma \ref{Lemma idempotent form}, we conclude that $\pi(W)$ is not idempotent modulo $n$, a contradiction.
This proves that the sequence $V\cdot (\mathop{\bullet}\limits_{i\in [1,r]} p_i^{[k_i-1]})$ is idempotent-product free modulo $n$.
Combined with \eqref{equation bigomega(K)} and \eqref{equation smallomega(K)}, we have that
\begin{equation}\label{equation I(S)geq in case prime power}
{\rm I}(\mathcal{S}_R)\geq |V\cdot (\mathop{\bullet}\limits_{i\in [1,r]} p_i^{[k_i-1]})|+1=(|V|+1)+\sum\limits_{i=1}^r (k_i-1)=
{\rm D}(R^{\times})+\Omega(n)-\omega(n).
\end{equation}

Now we assume that $n$ is a prime power or a product of pairwise distinct primes, i.e., either $r=1$ or $k_1=\cdots =k_r=1$ in \eqref{equation factorization of n}. It remains to show the equality ${\rm I}(\mathcal{S}_R)=
{\rm D}(R^{\times})+\Omega(n)-\omega(n)$ holds. We distinguish two cases.

\noindent \textbf{Case 1.} \ $r=1$ in \eqref{equation factorization of n}, i.e., $n=p_1^{k_1}$.

Take an arbitrary sequence $T$ of integers of length $|T|={\rm D}(R^{\times})+k_1-1={\rm D}(R^{\times})+\Omega(n)-\omega(n)$.
 Let $T_1=\mathop{\bullet}\limits_{\stackrel{a\mid T}{a\equiv 0 \pmod {p_1}}} a$ and $T_2=T\cdot T_1^{[-1]}$.
 By the Pigeonhole Principle, we see that either $|T_1|\geq k_1$ or $|T_2|\geq {\rm D}(R^{\times})$. It follows that either $\pi(T_1)\equiv 0 \pmod {p_1^{k_1}}$, or $\bar{1}\in\prod(\mathop{\bullet}\limits_{a\mid T_2} \bar{a})$. By Lemma \ref{Lemma idempotent form}, the sequence $T$ is not idempotent-product free modulo $n$, which implies that ${\rm I}(\mathcal{S}_R)\leq {\rm D}(R^{\times})+\Omega(n)-\omega(n)$. Combined with \eqref{equation I(S)geq in case prime power}, we have that ${\rm I}(\mathcal{S}_R)={\rm D}(R^{\times})+\Omega(n)-\omega(n).$

\noindent \textbf{Case 2.} \ $k_1=\cdots =k_r=1$ in \eqref{equation factorization of n}, i.e., $n=p_1p_2\cdots p_r$.

Then
\begin{equation}\label{equation bigomega=smallomega}
\Omega(n)=\omega(n)=r.
\end{equation}
Take an arbitrary sequence $T$ of integers of length $|T|={\rm D}(R^{\times})$. By the Chinese Remainder Theorem,
for any term $a$ of $T$ we can take an integer $a'$  such that for each $i\in [1,r]$,
\begin{equation}\label{equaiton tilde a}
a'\equiv\left\{ \begin{array}{ll}
1 \pmod {p_i} & \textrm{if $a\equiv 0 \pmod {p_i}$;}\\
a \pmod {p_i} & \textrm{otherwise.}\\
\end{array} \right.
\end{equation}
Note that $\gcd(a', n)=1$ and thus $\mathop{\bullet}\limits_{a\mid T}\bar{a'}\in \mathcal{F}(R^{\times})$. Since $|\mathop{\bullet}\limits_{a\mid T}\bar{a'}|=|T|={\rm D}(R^{\times})$, it follows that $\bar{1}\in \prod (\mathop{\bullet}\limits_{a\mid T}\bar{a'})$, and so there exists a {\bf nonempty} subsequence $W$ of $T$ such that $\prod\limits_{a\mid W}a'\equiv 1 \pmod {p_i}$ for each $i\in [1,r]$. Combined with \eqref{equaiton tilde a}, we derive that $\pi(W)\equiv 0\pmod {p_i}$ or $\pi(W)\equiv 1 \pmod {p_i}$, where $i\in [1,r]$.
By Lemma \ref{Lemma idempotent form}, we conclude that $\pi(W)$ is idempotent modulo $n$. Combined with \eqref{equation bigomega=smallomega}, we have that ${\rm I}( \mathcal{S}_R)\leq {\rm D}(R^{\times})={\rm D}(R^{\times})+\Omega(n)-\omega(n)$. It follows from \eqref{equation I(S)geq in case prime power} that ${\rm I}(\mathcal{S}_R)={\rm D}(R^{\times})+\Omega(n)-\omega(n)$, completing the proof. \qed

We close this paper with the following conjecture.

\begin{conj}
\  Let $n>1$ be an integer, and let $R=\mathbb{Z}_n$ be the ring of integers modulo $n$.
Then
${\rm I}(\mathcal{S}_R)={\rm D}(R^{\times})+\Omega(n)-\omega(n).$
\end{conj}

\bigskip

\noindent {\bf Acknowledgements}

\noindent
This work is supported by NSFC (grant no. 61303023, 11501561).

\end{document}